\documentclass[12pt,leqno]{article}
\usepackage{amsmath, amsthm, amsfonts, amssymb, color}
\usepackage{mathrsfs}
\newtheorem{thm}{Theorem}[section]

\newtheorem{lem}[thm]{Lemma}

\theoremstyle{definition}

\newcommand{\scr}[1]{\mathscr #1}
\definecolor{wco}{rgb}{0.5,0.2,0.3}

\numberwithin{equation}{section} \theoremstyle{remark}

\newcommand{\ua}{\uparrow}

\title{{\bf Gradient Estimate   on the Neumann  Semigroup  and Applications}\footnote{Supported in
 part by WIMCS and  NNSFC(10721091, 10771221, 10925106).}
}
\author{
{\bf Feng-Yu Wang$^{a),b)}$ and Lixin Yan$^{c)}$}\\
\footnotesize{$^{a)}$ School of Math. Sci. and Lab. Math. Com. Sys.,
Beijing Normal
University, Beijing 100875, China}\\
 \footnotesize{$^{b)}$ Department of Mathematics,
Swansea University, Singleton Park, SA2 8PP, UK}\\
\footnotesize{Email: wangfy@bnu.edu.cn;
F.Y.Wang@swansea.ac.uk}\\
\footnotesize{$^{c)}$ Department of Mathematics, Sun Yat-sen (Zhongshan) University, Guangzhou, 510275, China} \\
\footnotesize{Email: mcsylx@mail.sysu.edu.cn}\\
}
\begin{document}
\def\R{\mathbb R}  \def\ff{\frac} \def\ss{\sqrt} \def\B{\mathbf
B}
\def\N{\mathbb N} \def\kk{\kappa} \def\m{{\bf m}}
\def\dd{\delta} \def\DD{\Delta} \def\vv{\varepsilon} \def\rr{\rho}
\def\<{\langle} \def\>{\rangle} \def\GG{\Gamma} \def\gg{\gamma}
  \def\nn{\nabla} \def\pp{\partial} \def\EE{\scr E}
\def\d{\text{\rm{d}}} \def\bb{\beta} \def\aa{\alpha} \def\D{\scr D}
  \def\si{\sigma} \def\ess{\text{\rm{ess}}}
\def\beg{\begin} \def\beq{\begin{equation}}  \def\F{\scr F}
\def\Ric{\text{\rm{Ric}}} \def\Hess{\text{\rm{Hess}}}
\def\e{\text{\rm{e}}} \def\ua{\underline a} \def\OO{\Omega}  \def\oo{\omega}
 \def\tt{\tilde} \def\Ric{\text{\rm{Ric}}}
\def\cut{\text{\rm{cut}}} \def\P{\mathbb P} \def\ifn{I_n(f^{\bigotimes n})}
\def\C{\mathbb C}      \def\aaa{\mathbf{r}}     \def\r{r}
\def\gap{\text{\rm{gap}}} \def\prr{\pi_{{\bf m},\varrho}}  \def\r{\mathbf r}
\def\Z{\mathbb Z} \def\vrr{\varrho} \def\ll{\lambda}
\def\L{\scr L}\def\Tt{\tt} \def\TT{\tt}\def\II{\mathbb I}
\def\i{{\rm in}}\def\Sect{{\rm Sect}}\def\E{\mathbb E} \def\H{\mathbb H}
\def\M{\scr M}\def\Q{\mathbb Q} \def\texto{\text{o}} \def\LL{\Lambda}
\def\Rank{{\rm Rank}} \def\B{\scr B}
\def\T{\mathbb T}\def\i{{\rm i}} \def\ZZ{\hat\Z}\def\RN{\mathbb R^d}

\maketitle
\begin{abstract} We prove  the following sharp   upper bound for the 
gradient of the Neumann semigroup $P_t$ on a $d$-dimensional compact 
domain $\OO$ with boundary either   $C^2$-smooth or convex:

$$\|\nn P_t\|_{1\to \infty}\le \ff{c}{t^{(d+1)/2}},\ \ t>0,$$ 
where $c>0$ is a constant depending on the domain and
$\|\cdot\|_{1\to\infty}$ is the operator norm from $L^1(\OO)$ to $L^\infty(\OO)$.
This estimate implies a Gaussian type point-wise upper bound 
for the gradient of the Neumann heat kernel, which is applied to the study of the Hardy spaces,
Riesz transforms, and  regularity
 of solutions to the inhomogeneous Neumann problem on compact convex domains.
\end{abstract} \noindent

 \vskip 2cm

\section{Introduction}

 Let $\OO$ be a $d$-dimensional compact Riemannian manifold with boundary
$\pp \OO$. Let $N$ be the   inward unit normal vector field of $\pp
\OO$. If $\pp \OO$ is $C^2$-smooth, then the second fundamental form of $\pp \OO$ is defined as
$$\II(v_1, v_2):= -\<\nn_{v_1}N, v_2\>,\ \ \ v_1,v_2\in T\pp \OO,$$ where $T\pp \OO$ is the tangent space of $\pp \OO$.
We
call the boundary $\pp \OO$ (or the manifold $\OO$) convex if $\mathbb
I(v,v)\ge 0$ for all $v\in T\pp \OO.$

Let $P_t$ be the
Neumann semigroup generated by $\DD$, the Laplacian  on $\OO$ with the Neumann boundary condition. Let $p_t(x,y)$ be the Neumann heat kernel, which is the density of $P_t$ w.r.t. the Riemannian volume measure. For any $p,q\ge 1$, let $\|\cdot\|_{p\to q}$ denote the operator norm from $L^p(\OO)$ to $L^q(\OO)$.
When $\pp \OO$ is $C^2$-smooth, it is easy to prove the following uniform gradient estimate of $P_t$, which is important in potential analysis of the Neumann Laplacian as shown in Section 4 below.

\beg{thm} \label{T1.1} Let $\OO$ be a compact Riemannian manifold with    $C^2$-smooth boundary. Then there exists a constant $C>0$ such that

\beq\label {UG} \|\nn P_t\|_{1\to\infty}\le \ff{C}{t^{(d+1)/2}},\ \ t>0.\end{equation} Consequently, letting $\rr$ be the Riemannian distance on $\OO$,  one has

\beq\label{UH} |\nn p_t(\cdot, y)(x)|\le \ff{C}{t^{(d+1)/2}}\exp\Big[-\ff{\rr(x,y)^2}{ct}\Big],\ \ t>0, x,y\in \OO\end{equation} for some constants $C,c>0.$ \end{thm}

We remark that (\ref{UG}) is sharp (for short time) since the equality holds for the classical heat semigroup on $\RN$ and some constant $C>0.$  Since the boundary is $C^2$-smooth such that the second fundamental form is bounded below, the above theorem can be proved by using the reflecting Brownian motion and exponential moments of its local time (see Section 2).

When $\pp \OO$ is merely Lipschitzian such that the second fundamental form is not well defined,   the argument
 presented in Section 2 is no longer valid. Indeed, for general  manifolds with Lipschitzian boundary, even  the existence and uniqueness of the reflecting Brownian motion is  unknown. Nevertheless, for compact convex domains on $\RN$, the reflecting Brownian motion has been constructed by Bass and Hsu in \cite{BH90, BH91}, which enables us to
 derive the above gradient estimates.

\begin{thm} \label{th1.1} Let $\Omega$ be a compact convex domain in $\RN$.
Then $(\ref{UG})$ and $(\ref{UH})$ hold for some constants $C,c>0$ and $\rr(x,y)=|x-y|.$
\end{thm}

The proofs of these two theorems will be given  in the next two sections respectively.  In Section 4, we introduce some  applications of our results
to   the study of the Hardy spaces,
Riesz transforms associated to the Neumann Laplacian, and  regularity
 of solutions to the inhomogeneous Neumann problem on compact convex domains.

\section{Proof of Theorem \ref{T1.1}}

We first observe that for the proof of (\ref{UG}) it suffices to consider $t\in (0,1].$ Let $\d x$ denote the Riemannian volume measure on $\OO$.
Since $\nn P_t f=\nn P_t \hat f$, where $\hat f:=f-\int_\OO f(x)\d x$, we may assume that $f$ itself has zero integral on $\OO$. Let $\ll_1 (>0)$ be the first Neumann eigenvalue on $\OO$. We have

\beq\label{LL} \|P_t f^2\|_2\le \e^{-\ll_1t}\|f\|_2,\ \ t>0.\end{equation}
On the other hand,
 by taking e.g. $s=t$ in \cite[Corollary 1.4]{W10b}, we obtain

$$|B(x,\ss t)|\cdot |P_t f(x)|\le C_1\int_{B(x,\ss t)} P_{2t}|f|(y)\d y\le C_2\|f\|_1, \ \ t\in (0,1]$$ for some constant $C_1>0$, where $B(x,\ss t):=\{y\in \OO: \rr(x,y)\le \ss t\}$ with volume  $|B(x,\ss t)|$.  Since $\OO$ is compact, we have $|B(x,\ss t)|\ge ct^{d/2}$ for some constant $c>0$
 and all $t\ge 0, x\in \OO$. Therefore,

\beq\label{UU} \|P_t\|_{1\to\infty} \le \ff{C_2} {t^{d/2}},\ \ t\in (0,1]\end{equation}  for some constant $C_2>0$.
Now, if (\ref{UG}) holds for $t\le 1$,  then  (\ref{LL}) and (\ref{UU}) yield

\beg{equation*}\beg{split} \|\nn P_t f\|_{\infty}&= \|\nn P_{1/2}P_{t-1}P_{1/2}f\|_{\infty}\le C_3 \|P_{t-1}P_{1/2}f\|_2\\
&\le C_3\e^{-(t-1)\ll_1} \|P_{1/2} f\|_2\le C_4 \e^{-\ll_1 t}\|f\|_1,\ \ t>1\end{split}\end{equation*}
 for some constants $C_3,C_4>0$. Therefore, (\ref{UG}) also holds for $t>1.$

From now on, we assume that $t\le 1.$  Let $K,\si$ be two constants such that $\Ric\ge -K$ holds on $\OO$ and $\II\ge -\si$ holds on $\pp\OO.$ By \cite[Theorem 1.1(7)]{W10a}, we have

\beq\label{2.1} |\nn P_t f|^2\le \ff{2K^2}{(1-\e^{-2Kt})^2} \big(P_t f^2 -(P_t f)^2\big) \E \int_0^t\e^{2\si l_s-2Ks}\d s,\  f\in C(M),\end{equation} where $(l_s)_{s\ge 0}$ is the local time on $\pp \OO$ for the reflecting Brownian motion on $\OO$. Next, according to \cite[Proof of Lemma 2.1]{W05}, there exists a constant $c>0$ such that

$$\E\e^{\si l_s}\le \e^{c (s+\ss s)},\ \ \ s\le 1.$$ Therefore, it follows from (\ref{2.1}) that

$$\|\nn P_t f\|_\infty^2 \le \ff{C_5}{t} \|P_t f^2\|_\infty,\ \  f\in C(M)$$ holds for some constant $C_5>0.$  Replacing  $f$ by $P_tf$ and using (\ref{UU}), we arrive at

$$\|\nn P_{2t} f\|_\infty \le\ff{C_6}{t^{(d+1)/2}}\|f\|_1,\ \ f\in C(M)$$ for some constant $C_6>0$. This proves (\ref{UG}) for $t\le 2.$

Finally, since (\ref{UG}) is equivalent to

$$\sup_{x,y\in\OO} |\nn p_t(\cdot, y)(x)|\le \ff{C}{t^{(d+1)/2}},\ \ t>0,$$ the inequality (\ref{UH}) follows from the self-improvement property as in \cite[Theorem 4.9]{CS}.

\section{Proof of Theorem~\ref{th1.1} }

We shall make use of the reflecting Brownian motion determined as
solutions to the Skorokhod equation

\begin{equation}\label{E}
X_t= x+ W_t + \int_0^t N(X_s) \d l_s, \ \ t\geq
0,\end{equation}

\noindent
where $x\in \OO$, $W_t$ is a $d$-dimensional Brownian motion on
a complete filtrated probability space $({\mathscr E}, \mathscr F_t, \mathbb P)$, $X_t$
is a continuous adapted process on $\Omega$, and $l_t$ is a predictable
continuous increasing process with $l_0=0$ which increases only when
$X_t\in \partial \Omega$. If $(X_t, l_t)$ solves (\ref{E}) for some
$d$-dimensional Brownian motion $W_t$, we call $X_t$ the reflecting
Brownian motion on $\Omega$starting from $x$, and call  $l_t$ its local time on $\partial \Omega$.

\begin{lem}\label{L1}
 For any   $d$-dimensional Brownian motion $W_t$
and   any $x\in \Omega$, $(\ref{E})$ has a  unique solution.
\end{lem}

\begin{proof}
(a) Uniqueness. Let $(X_t, l_t)$ and $(\tilde
X_t, \tilde l_t)$ be two solutions to (\ref{E}). By the It\^o formula,

$$
|X_t-\tilde X_t|^2 = 2\int_0^t\langle X_s-\tilde X_s, N(X_s)\rangle \d l_s + 2\int_0^t \langle
{\tilde X}_s-X_s, N({\tilde X}_s)\rangle \d{\tilde l}_s.
$$

\noindent
By the convexity of $\partial \Omega$ we
see that $\langle y-z, N(y)\rangle\leq 0$ if $y\in \partial \Omega$ and $z\in \Omega$.
Moreover, since $ d l_s=0$ (correspondingly, $ d {\tilde l}_s=0$) for
$X_s\notin \partial \Omega$ (correspondingly, $\tilde X_s\notin \partial \Omega$), we
conclude that $|X_t-\tilde X_t|^2=0$ for all $t\geq 0.$

(b) Existence. By using the regularity of the associated Dirichlet
form,  \cite[Theorem 4.4]{BH91} ensures the existence of a
reflecting Brownian motion.  According to \cite[Theorem 1]{BH90},
this reflecting Brownian motion  solves the Skorokhod equation
(\ref{E}) for some $d$-dimensional Brownian motion $\tilde W_t$,  i.e.
there exists $(\tilde X_t, \tilde l_t)$ such that $\tilde X_t$ is a
continuous adapted process on $\Omega$, $\tilde l_t$ is a predictable
continuous increasing process with $\tilde l_0=0$ which increases only
when $\tilde X_t\in\partial \Omega$, and

\begin{equation}\label{E'}
\tilde X_t= x +\tilde W_t + \int_0^t N(\tilde X_s) \d {\tilde l}_s,\
\ t\geq 0
\end{equation}

\noindent
 holds. We aim to prove that for any
$d$-dimensional Brownian motion $W_t$, the equation has a solution.
Due to (a), this  follows from the  Yamada-Watanabe criterion \cite{YW} (cf. the proof of
\cite[Theorem 5.9]{BBC}). For readers' convenience, we present below a brief proof.   By the uniqueness of solutions to (\ref{E'}),
$(\tilde X, \tilde l)$ is determined by $\tilde W$; that is, there exists a
  measurable function

  $$
  F=(F_1,F_2): C([0,\infty); \RN)\to C([0,\infty);\, \Omega)\times C_I([0,\infty);
  [0,\infty))$$

  \noindent
  for $\sigma$-fields induced by the topology of locally uniform
  convergence, where $C_I([0,\infty); [0,\infty))$ is the space of
  all non-negative continuous  increasing functions on $[0,\infty)$ with initial data $0$, such that

  $$
  (\tilde X,\tilde l)= F(\tt W)=(F_1(\tt W), F_2(\tt W)).
  $$

  \noindent
  So, letting $\mu$ be the Wiener measure
  on $C([0,\infty); \RN)$ (i.e. the distribution of $\tt W$), we have

  $$
  F_1(\omega)_t = x +\omega_t +\int_0^t N(F_1(\omega)_s) \d F_2(\omega)_s,\ \
  t\geq 0
  $$

  \noindent
  for $\mu$-a.e. $\omega\in C([0,\infty);\RN).$ Therefore,
  for any $d$-dimensional Brownian motion $W$, which has the same distribution $\mu$,

$$
F_1(W)_t = x +W_t +\int_0^t N(F_1(W)_s) \d F_2(W)_s,\ \
  t\geq 0
  $$

  \noindent
  holds a.s. This means that $(X,l):= (F_1(W), F_2(W))$
  solves the equation (\ref{E}).
\end{proof}

\medskip

\begin{lem} \label{L2}
Let $X_t^x$ be the unique solution to
$(\ref{E})$ for $X_0=x\in \Omega.$ For $f\in \B_b(\OO)$, the class of all bounded measurable functions on $\OO$, let $P_t f(x):= u(t,x)$ solve the Neumann heat equation

\beq\label{N} \pp_t u(\cdot,x)(t)=\DD u(t,\cdot)(x),\ Nu(t,\cdot)|_{\pp \OO}=0, u(0,\cdot)=f.\end{equation}
Then

$$
P_t f(x)= \mathbb E f(X_{2t}^x),\ \ t\ge 0, x\in\OO.
$$
\end{lem}

\begin{proof} Since $u(t,x):= P_t f(x)$ satisfies (\ref{N}) and
$X_t^x$ satisfies (\ref{E}), by the It\^o formula,  
the process

$$
\Phi(s):=u(t-s, X_{2s}^x)=P_tf(x)+ \int_0^s 2\langle\nabla u(t-s,\cdot)(X_{2s}^x), \d
W_s\rangle,\ \ s\in [0,t]
$$

\noindent
is a martingale. In particular,

$$
P_tf(x)=  \mathbb E \Phi(0)=\mathbb E \Phi(t)= \mathbb E f(X_{2t}^x).
$$
\end{proof}

\medskip

\begin{proof} [Proof of Theorem~\ref{th1.1}]
As explained in the proof of Theorem \ref{T1.1}, it suffices to prove (\ref{UG}) for $t\in (0,1]$. To this end,
we shall make use of the reflecting Brownian motion introduced
above.

Let $P_t^0f(x)= \mathbb E f(X_t^x).$ To estimate the gradient of $P_t f$,
let $y\ne x$ be two points in $\Omega$, and let $(X_t^x, l_t^x)$ and
$(X_t^y, l_t^y)$ solve (\ref{E}) with $X_0=x$ and $y$ respectively.
Then, as explained in the proof of Lemma \ref{L1}, the convexity of
$\Omega$ implies that

$$|X_t^x-X_t^y|\leq |x-y|,\ \ t\geq 0.$$ Thus,

$$
\frac{|P_{t}^0 f(y)-P_t^0 f(x)|}{|x-y|} \leq
\frac{\mathbb E|f(X_{t}^x)-f(X_{t}^y)|}{|x-y|}\leq
\mathbb E\Big(\frac{|f(X_{t}^x)-f(X_{t}^y)|}{|X_{t}^x-X_t^y|}\Big).
$$

\noindent
Letting
$y\to x$ we arrive at

\begin{equation}\label{GG}
|\nabla P_t^0 f|\leq P_t^0|\nabla f|,\ \ t\geq
0.
\end{equation}
Due to an argument of Bakry-Emery \cite{BE}, this
implies that

\begin{equation}\label{P}
t|\nabla P_t^0 f|^2\leq P_t^0 f^2-(P_t^0 f)^2,\ \ t> 0,
f\in \mathscr B_b(\Omega).
\end{equation}

\noindent
Indeed, for a smooth function $g$ on
$\Omega$ satisfying the Neumann boundary condition, (\ref{E}) and the
It\^o formula for $g(X_t^x)$ imply that

\begin{equation}\label{A}
P_t^0 g={1\over 2}  \int_0^t P_s \Delta g \d s.
\end{equation}

\noindent
Since $P_t =P_{2t}^0$, by Lemma \ref{L2} we have

$$\ff{\d}{\d t}P_{2t}^0g =  \DD P_{2t}^0g.$$  Combining this with (\ref{A}) we obtain

\beq\label{*W}
{\d\over \d s} P_s^0 g= {1\over 2}  \Delta P^0_sg = {1\over 2}  P^0_s \Delta g.
\end{equation}

\noindent
Hence, it follows from (\ref{GG})  and the Jensen inequality that

\begin{eqnarray}\label{ee}
t\big|\nabla P^0_tf\big|^2 =\int_0^t\big|\nabla P^0_s \big( P^0_{t-s}f\big)\big|^2\d s
  \leq \int_0^t  P^0_s  \big|\nabla P^0_{t-s}f \big|^2 \d s.
\end{eqnarray}

\noindent
(\ref{*W}) also implies  that  $
{\d\over \d s} \Big( P_s^0 (P_{t-s}^0 f)^2\Big) =   P^0_s  \big|\nabla P^0_{t-s}f \big|^2.
$  This, together with (\ref{ee}), shows that

 $$
 P_t^0 f^2-(P_t^0f)^2 =\int_0^t \frac{ \d }{ \d s} P_s^0 (P_{t-s}^0f)^2\d s=
\int_0^ tP_s^0|\nabla P_{t-s}^0f|^2\d s\geq t |\nabla P_t^0 f|^2.
$$

\noindent
So,
(\ref{P}) holds. Combining (\ref{P}) with the known uniform upper
bound of the Neumann heat kernel on convex domains (see e.g. \cite[Theorem 3.2.9]{Da})

\begin{equation}\label{H}
\|P_t f\|_\infty \leq C t^{-d/2} \int_\Omega |f|(x)\d x,\
\ t\in (0,1], f\in \mathscr B_b(\Omega),
\end{equation}
we conclude that

\begin{eqnarray*}
\|\nabla P_t f\|_\infty&=& \|\nabla P_t^0 (P_{t/2}f)\|_\infty
 \leq {1\over \sqrt{t}}\sqrt{ P_t^0\big[(P_{t/2} f)^2\big]}\\
&\leq&{C\over \sqrt{t}}
\|P_{t/2}f\|_\infty
 \leq  C t^{-(d+1)/2} \int_\Omega |f|(x) \d x,\ \ f\in
\mathscr B_b(\Omega)
\end{eqnarray*}

\noindent
holds for some $C>0$ and all $t\in (0,1].$ This
implies the desired gradient estimate for $t\in (0,1].$
\end{proof}

\bigskip

\section{Applications}
\setcounter{equation}{0}

Throughout this section, we let $\OO$ be a compact convex domain in $\RN$, and let    $\OO^\circ$ be the interior of $\OO$. 
It is well known that the generator $(\DD,\D(\DD))$ of
$P_t$ in $L^2(\OO)$ is a negatively definite self-adjoint operator with 
discrete spectrum.  Let $\DD_N=-\DD$, which is thus a positive definite 
self-adjoint operator such that $P_t =\e^{-\DD_N t},\ t\ge 0.$

\subsection{The Hardy spaces on compact convex domains}

For $0<p\leq 1$ we let $h^p(\RN)$ denote the classical (local)
Hardy space in $\RN$. We consider two versions of this space adapted
to  $\Omega$.
Let ${\mathscr D}(\Omega)$ denote the space of $C^{\infty}$ functions with
compact support in $\Omega^\circ$, and let ${\mathscr D}'(\Omega)$ denote  its dual,
the space of distributions on $\Omega^\circ$. The first adaptation of the local
Hardy space to $\Omega$, denoted by $h^p_r(\Omega)$, consists of
elements of ${\mathscr D}'(\Omega)$ which are the restrictions to $\Omega^\circ$ of
elements of $ h^p(\RN)$. That is, for $0<p\leq 1$ we set
\begin{equation*} \beg{split}
h^p_r(\Omega)&:=\big\{f\in{\mathscr S}'(\RN):
\mbox{there exists $F\in h^p(\RN)$ such that $F|_{\Omega^\circ}=f$} \big\}\\
&=h^p(\RN)/\big\{F\in h^p(\RN):\,
\mbox{\rm $F=0$ in $\Omega^\circ$}\big\},
\end{split}\end{equation*}
\noindent which is equipped with the quasi-norm
$$
\|f\|_{h^p_r(\Omega)}:=\inf\big\{\|F\|_{h^p(\RN)}:\,
\mbox{$F\in h^p(\RN)$ such that $F|_{\Omega^\circ}=f$}\big\}.
$$
The second adaptation of the local Hardy space to $\Omega$, denoted by $h^p_z(\Omega)$,
consists of distributions   in $\Omega^\circ$ with   extension  by zero to $\RN$ belonging to
$h^p(\RN)$. More specifically, for $0<p\leq 1$,
$$h^p_z(\Omega)
:= h^p(\RN)\cap\big\{ f\in h^p(\RN):
\mbox{\rm  $f=0$ in $\Omega^c$}\big\}/\big\{ f\in h^p(\RN):
\mbox{\rm $f=0$ in $\Omega^\circ$}\big\},
$$ where $\OO^c:=\RN\setminus \OO.$ We can identify $ h^p_z(\Omega)$ with a set of distributions in
${\mathscr D}'(\Omega)$ which, when equipped with the natural quotient norm,
becomes a subspace of $h^p_r(\Omega)$ (see \cite{CKS}
for more details). Obviously,
we have that $h_z^p(\Omega)\subseteq  h_z^p(\Omega)$ whenever  $ d/(d+1)<p\leq 1.$

Given a function $f\in L^2(\Omega)$, consider the following
local version of the non-tangential maximal operator associated
with the heat semigroup generated by the operator $\Delta_N$:

$$
N_{\rm loc,\,\Delta_N}f(x)
:=\sup_{y\in \Omega,\,|y-x|<t\leq 1}|\e^{-t^2\Delta_N}f(y)|,\quad x\in\Omega.
$$

\noindent
For $0<p\leq 1$, the space  $h_{\Delta_N}^p(\OO)$ is then defined as the completion
of $L^2(\Omega)$ in the quasi-norm
\begin{eqnarray}\label{RB-D1}
\|f\|_{h_{\Delta_N}^p(\OO)}:=\|N_{\rm loc,\,\Delta_N}f\|_{L^p(\Omega)}.
\end{eqnarray}

\noindent

Based on the gradient estimate (\ref{UH})   in
Theorem~\ref{th1.1} and the adapted atomic theory of Hardy spaces $h_{\Delta_N}^p(\OO)$,
the following  result is  obtained in \cite{DHMMY}, see also  
\cite{AR, CKS, MM, DHMMY} and references within for  discussions on the Hardy spaces $h_r^p(\Omega)$.
\begin{thm}[\cite{DHMMY}]\label{th3.1}  Let $\Omega$ be a compact convex domain in $\RN$. Then
$h_z^p(\Omega)= h_{\Delta_N}^p(\Omega)$ for $ d/(d+1)<p\leq 1.$
\end{thm}

 \medskip



 \medskip

\subsection{Riesz transforms  on bounded convex domains}
Consider the generalised Riesz
transform  $T=\nabla \Delta_N^{-1/2}$ associated to the Neumann Laplacian $\Delta_N$,
defined by
$$
Tf={2\over  \sqrt{\pi}}\int_0^{\infty}
\nabla \e^{-s\Delta_N}f{ds\over \sqrt{s}}.
$$
The operator $ \nabla \Delta_N^{-1/2}$ is bounded on $L^2(\Omega).$
In \cite{CD}, it is shown that the operator $\nabla \Delta_N^{-1/2}$ is of weak (1,1)
by making use of  (\ref{H}), hence by interpolation,
is bounded on $L^p(\Omega)$ for $1<p\leq 2.$ For the case $p>2,$ according to \cite{ACDH},
the following assertions are equivalent:

\smallskip

(1) For all $p\in (2, \infty)$, there exists $C_p$ such that

$$
\big\| |\nabla \e^{-t\Delta_N}| \big\|_{p\to p}\leq {C\over \sqrt{t}}, \ \ \ \forall t>0.
$$

(2) The Riesz transform  $ \nabla \Delta_N^{-1/2}$ is bounded on $L^p(\Omega) $ for $p\in (2, \infty)$.

\medskip

In terms of the gradient estimate (\ref{UH})   in
Theorem~\ref{th1.1}, it deduces the following theorem.

\medskip

\begin{thm}\label{th3.2} Let $\Omega$ be a compact convex domain in $\RN$.
Let $ T=\nabla \Delta_N^{-1/2}$ be the Riesz transform associated to the Neumann Laplacian $\Delta_N$
on $\Omega$. Then
the operator $\nabla \Delta_N^{-1/2} $  is  bounded  on $L^p(\Omega)$ for all $1<p<\infty$.

If  $p=1$, then the operator $\nabla \Delta_N^{-1/2} $ is also of weak type $(1,1).$

\end{thm}

 \medskip

By (\ref{UH}) and (\ref{H}), we can also extend
the Riesz transform $\nabla \Delta_N^{-1/2} $    to a bounded operator  from $h^p_{\Delta_N}(\Omega)$ into
$L^p(\Omega)$ for $0<p\leq 1$. Hence by Theorem~\ref{th3.1},  it can be
extended to a bounded operator  from
  $h^p_{z}(\Omega)$ into
$L^p(\Omega)$ for ${d/(d+1)}<p\leq 1$.
The proof is similar  to that of  
 \cite[Theorem 4.2]{DHMMY}.  
\bigskip

\subsection{Regularity of solutions to the Neumann problem}
Define the Neumann Green 
operator ${\mathbb{G}}_N$ as the solution operator
$C^{\infty}(\Omega)\ni f\mapsto u={\mathbb{G}}_N(f)\in W^{1,2}(\Omega)$
for the Neumann problem
\begin{eqnarray}\label{e3.7}
\left\{
\begin{array}{rlll}
\Delta u&=f &{\rm in}&\Omega^\circ\\[4pt]
N u&=0 &{\rm on}&\partial\Omega,
\end{array}
\right.
\end{eqnarray}
\noindent where it is also assumed that $\int_{\Omega}f=0$ and the solution
is normalized by requiring that $\int_{\Omega}u=0$.
\medskip

\noindent
\subsubsection{Estimate for the gradient of Green potential}  Let $W^{s,p}(\Omega)$ stand  for the
Sobolev space of functions in $L^p(\Omega)$ with distributional derivatives of order
$s$ in $L^p(\Omega)$. By $L^p_{\bot}(\Omega)$ and $W^{s,p}_{\bot}(\Omega)$ we denote
the subspaces of functions $f$ in $L^p (\Omega)$ and $W^{s,p} (\Omega)$
subject to $\int_{\Omega}f(x) \d x=0.$

Recently,  V. Maz'ya proved the following result.

\begin{thm} [\cite{M}] \label{XX} Let $\Omega$ be a convex domain in ${\mathbb R}^d$.
Let $f\in L^q_{\bot}(\Omega)$ with a certain $q>d$. Then there exists a constant $c$
depending only on $d$ and $q$ such that  the solution
$u\in W^{1,2}_{\bot}(\Omega)$ of the problem $(\ref{e3.7})$ satisfies the estimate

\begin{eqnarray}\label{e3.4}
\|\nabla u\|_{L^{\infty}(\Omega)}\leq c(d,q)
C_{\Omega}^{-1} |\Omega|^{(q-d)/qd} \|f\|_{L^q(\Omega)}.
\end{eqnarray}
\end{thm}

\medskip

Note that by Theorem~\ref{th1.1},  we can   give  a simple proof of
Theorem~\ref{XX}.  Indeed, based on the gradient estimate (\ref{UH}) of $p_t$  in
Theorem~\ref{th1.1}, we have

$$
\big|\nabla G_N(x,y)\big|\leq {C\over |x-y|^{d-1}}, \ \ \ \ \ \forall x,y\in \Omega,
$$

\noindent
where $G_N $ is the Green function for the Neumann semigroup on $\Omega$, and hence by
a standard argument, estimate (\ref{e3.4}) follows readily.




\bigskip

\noindent
\subsubsection{ Estimate for the second-order derivatives for Green potential} 
Let  ${\mathbb{G}}_N$ be  the Neumann Green operator
for the Neumann problem in (\ref{e3.7}).
   Note that  $L^2$-boundedness of the mappings
$$
f\mapsto{\partial^2{\mathbb{G}}_N(f)\over\partial x_i\partial x_j},
\qquad 1\leq i, j\leq d,
$$
\noindent
    has been known since the mid 1970's (\cite{Gr}), but optimal
$L^p$ estimates, valid in the range $1<p\leq 2$, have only been proved
in the 1990's by Adolfsson and D. Jerison \cite{AJ}.
 It should be mentioned that the
aforementioned $L^p$ continuity of two derivatives on Green potentials may fail
in the class of Lipschitz domains for any $p\in(1,\infty)$ and in the class
of convex domains for any $p\in(2,\infty)$ (see  \cite{AJ}
for counterexamples; recall that every
convex domain is Lipschitz).

  A natural question is to study the regularity of the Neumann Green operator
when
the $L^p$-scale is replaced by the scale of Hardy spaces, $H^p$, for $0<p\leq 1$.
Recently, X. Duong, S. Hofmann, D. Mitrea, M. Mitrea and the second named  author of
  this article gave
a  solution to the conjecture made by D.-C. Chang, S. Krantz and E.M. Stein(\cite{CKS})
regarding the regularity of Green operators for the   Neumann
problems on $h^p_r(\Omega)$ and $h^p_z(\Omega)$, respectively, for all
${d\over d+1}<p\leq 1$, and this range of $p's$ is sharp (see \cite{DHMMY}).

\medskip

\begin{thm}[\cite{DHMMY}]\label{th3.4}
Let $\Omega$ be a compact, simply connected,  convex domain in
$\RN$ and recall that ${\mathbb{G}}_N$ stands for the
Green operator associated with the inhomogeneous Neumann problem
$(\ref{e3.7})$. Then the operators
\begin{eqnarray}\label{e3.8}
{\partial^2{\mathbb{G}}_N\over\partial x_i\partial x_j},\qquad i,j=1,\dots,d,
\end{eqnarray}
\noindent originally defined on
$\{f\in C^{\infty}(\OO):\,\int_{\Omega}f\,\d x=0\}$, extend
as bounded linear mappings from $h^p_{z}(\Omega)$ to $h^p_r(\Omega)$
whenever ${d/(d+1)}<p\leq 1$.

If $p=1,$ then the operators
${\partial^2{\mathbb{G}}_N\over \partial x_i\partial x_j}$,
$i,j=1,\dots, d,$ are also of weak type $(1,1)$. Hence by interpolation,
they can be extended to bounded operators on $L^p(\Omega)$ for $1<p\leq 2$.
\end{thm}

The proof of this theorem is obtained
by using  suitable estimates for singular
integrals with non-smooth kernels and an optimal on-diagonal heat kernel
estimate. We should mention that the gradient estimate (\ref{UH})   played a major role in
 the proof of this theorem, regarding the regularity of Green operators for the   Neumann
problems on $h^p_r(\Omega)$ and $h^p_z(\Omega)$, respectively, for all
${d\over d+1}<p\leq 1$.
 For the detail of the proof, we refer the reader to \cite{DHMMY}.

\vskip 1cm

 \noindent
{\bf Acknowledgments.}  L.X. Yan would like to thank
X.T. Duong, S. Hofmann, D. Mitrea and  M. Mitrea  for helpful discussions.

\end{document}